\documentclass[10pt]{article}

  \usepackage[all]{xy}
 \usepackage{graphicx}
 \usepackage{units}
 
 \usepackage{enumerate}
 
\usepackage{amssymb, amsmath, pict2e}
\usepackage{pdfsync}

\usepackage{rotating}

\usepackage{multirow}

  \usepackage[normalem]{ulem}
  \usepackage{cancel}

\usepackage{color}
\usepackage[usenames,dvipsnames,svgnames,table]{xcolor}

\title{Unit Interval Orders of Open  and Closed Intervals}

\author{Alan Shuchat \\  
\small Department of Mathematics \\ 
\small Wellesley College \\
\small Wellesley, MA 02481  \\ 
\small \texttt {ashuchat@wellesley.edu} \and
Randy Shull \\ 
\small Department of
Computer Science        \\ 
\small Wellesley College
\\ 
\small Wellesley, MA 02481   \\
\small \texttt{rshull@wellesley.edu}
\and Ann N.\ Trenk \\ 
\small Department of Mathematics \\ 
\small Wellesley College \\
\small Wellesley, MA 02481  \\ 
\small \texttt {atrenk@wellesley.edu}\\
}

\date{January 12, 2015}

\newtheorem{theorem}{Theorem}
\newtheorem{defn}[theorem]{Definition}

\newtheorem{remark}[theorem]{Remark}

\newtheorem{prop}[theorem]{Proposition}

\newcommand{\qed}{\mbox{$\Box$}}

\def\reals{{\mathbb R}}
\def\integers{{\mathbb Z}}
\newcommand{\fourone}{\mbox{${\bf 4}+{\bf 1}$}}

\newcommand{\threeone}{\mbox{${\bf 3}+{\bf 1}$}}
\newcommand{\threeoneone}{\mbox{${\bf 3}+{\bf 1}+{\bf 1}$}}
\newcommand{\twotwo}{\mbox{${\bf 2}+{\bf 2}$}}

\begin{document}

  \maketitle

\bibliographystyle{plain}

\begin{center} {\sl ABSTRACT} \end{center}

\begin{quotation}
A poset $P = (X,\prec)$ is a unit OC interval order if there exists a representation that   assigns an open or closed real interval $I(x)$ of unit length to each $x \in P$ so that $x \prec y$ in $P$ precisely when each point of $I(x)$ is less than each point in $I(y)$.  
In this paper we give a forbidden poset characterization of the class of unit OC interval orders and an  efficient algorithm for recognizing the class.  The algorithm takes a poset $P $ as input and either produces     a representation or returns a forbidden poset induced in $P$.
\end{quotation}


\section{Introduction}
In \cite{RaSz13}, the authors introduce the class of unit intersection graphs of real intervals from the set   $\{(a,a+1): a \in \reals\} \cup \{[a,a+1]: a \in \reals\}$.  This class, which they call  ${\cal{U}}^\pm,$ includes $K_{1,3}$ and thus  is larger than the set of unit interval graphs.  They give a forbidden subgraph characterization of the class.  Joos~\cite{Jo13} gives a forbidden graph characterization of the class of unit mixed interval graphs in which intervals may be open, closed, or half-open.  Independently in \cite{ShShTrWe15} we also give a forbidden characterization of this class and additionally an algorithm that produces a representation when a  graph has none of the forbidden subgraphs.  Le and Rautenbach~\cite{LeRa13} study a version of unit mixed interval graphs in which the endpoints of intervals must be integers.

In this paper, we return to the original class in which intervals may be open or closed (but not half-open) and study it from the perspective of ordered sets.  
  We give a forbidden poset characterization as well as a quadratic-time recognition and realization algorithm.    Our algorithm accepts any poset as input and either returns a representation of it or a forbidden subposet.

\subsection{Preliminaries}

The posets $P = (X, \prec)$  in this paper are  irreflexive, and we write $x \parallel y$ when elements (points)  $x, y \in X$ are incomparable.  For more background on ordered sets, see \cite{Tr92}, and for interval graphs see \cite{Go80}.
We denote the left and right endpoints of a real interval $I(v)$ respectively by $L(v)$ and $R(v)$.  An interval is \emph{open} if it does not contain its endpoints and \emph{closed} if it does.   

 \begin{defn} {\rm
  A poset $P = (X, \prec)$  is an \emph{interval order} if each $x \in X$ can be assigned a closed real interval $I(x)$ so that $x \prec y$ if and only if $R(x) < L(y)$.  The set of intervals ${\cal I} = \{I(x) : x \in X \}$ is a (\emph{closed}) \emph{interval representation} of $P$.  If in addition, all intervals in the representation have the same length, then $P$ is a \emph{unit interval order} and ${\cal I} $ is a  \emph{unit} (\emph{closed}) \emph{interval representation} of $P$.
  }
  \end{defn}
  
  It is well-known (e.g., see  Lemma 1.5 in \cite{GoTr04}) that a set of closed intervals $\cal I$ representing a poset $P$ can be transformed into another set ${\cal I'}$ of closed intervals also representing $P$ so that the endpoints of the intervals in ${\cal I'}$ are distinct.  Furthermore, this can be accomplished so that  if $\cal I$ is a unit representation of $P$, so is $ {\cal I'}$.
   It follows that  the class of interval orders is  the same if   representations consist  entirely of closed intervals or entirely of open intervals.  Likewise, the class of unit interval orders is the same whether  representations consist entirely of  closed intervals or entirely of open intervals.   The situation is different if both open and closed intervals are allowed.    

 \begin{defn}  {\rm   Let ${\cal R}$ be the set of all intervals on the real line that are either open or closed.  An \emph{OC interval representation} of a poset $P = (X, \prec)$ is an assignment of an interval $I(x) \in \cal R$ to each  $x \in X$, so that  $x \prec y$ if and only if each point in $I(x)$ is less than each point in $I(y)$.   We say a poset is an \emph{OC interval order} if it has an OC interval representation ${\cal I}$.  If in addition, all intervals in ${\cal I}  = \{I(x) : x \in X \}$ have the same length, then $P$ is a \emph{unit OC interval order} and we call ${\cal I}$ a \emph{unit} \emph{OC interval representation} of $P$.}
 \label{def-rep}
 \end{defn}  
 
 The equivalence of (1) and (2) in the following proposition is well-known and was shown by Fishburn \cite{Fi70}.  The equivalence of (1) and (3)   is part of a similar result given for interval graphs in \cite{RaSz13}.  A sketch of the latter equivalence is included here for completeness.

 \begin{prop}
 \label{prop-OC}
 The following are equivalent for a poset $P$:
 \begin{enumerate}
 \item  $P$ is an interval order.
 \item  $P$ contains no induced $\twotwo$.
 \item  $P$ is 
  an OC interval order.
  \label{OC-int-prop}
  \end{enumerate}
  \end{prop}
    \begin{proof}  By definition, interval orders are OC interval orders.  For the converse,
    let $P= (X,\prec)$ have an OC interval representation $\cal I$.       For each pair of incomparable elements $u,v \in X$, pick a point $x_{uv}$ in the set $I(u) \cap I(v)$.  For each vertex $v \in V$, define $I'(v) = [\min\{x_{uv}: u \parallel v\}, \max\{x_{uv}: u \parallel v\}]. $  One can check that the intervals $I'(v)$ give a closed interval representation of $P$, so $P$ is an interval order.  
    \qed
    \end{proof}

    Proposition~\ref{prop-OC} shows that the class of interval orders remains unchanged even when both open and closed intervals are allowed. However, we will see that  the class of unit   interval orders is enlarged when both open and closed intervals are allowed.

  An interval order is \emph{proper} if it has a closed interval representation in which no interval is  properly contained in  another.  By definition, the class of unit interval orders is contained in the class of proper interval orders, and in fact, the classes are equal, as implied by the work of  \cite{ScSu58} and written explicity in terms of graphs in \cite{Ro69}.  When both open and closed intervals are permitted in a representation, we must refine the notion of \emph{proper} in order to maintain  the inclusion of the unit class in the proper class.
  
 \begin{defn} {\rm  An interval $I(u)$ is \emph{strictly contained} in an interval $I(v)$ if $I(u) \subset I(v)$ and they do not have identical endpoints.
 An OC interval representation ${\cal I}$ is  \emph{strict} if no interval is strictly contained in another, i.e., if
  the only proper inclusions allowed are between intervals with the same endpoints. }
  \end{defn}

We will show in Theorem~\ref{big-thm}
that the classes of unit OC interval orders and strict OC interval orders are equal and give a forbidden poset characterization of this class.

\subsection{Forbidden posets}

  \begin{figure}
\centering
  {\includegraphics[height=1.35in]{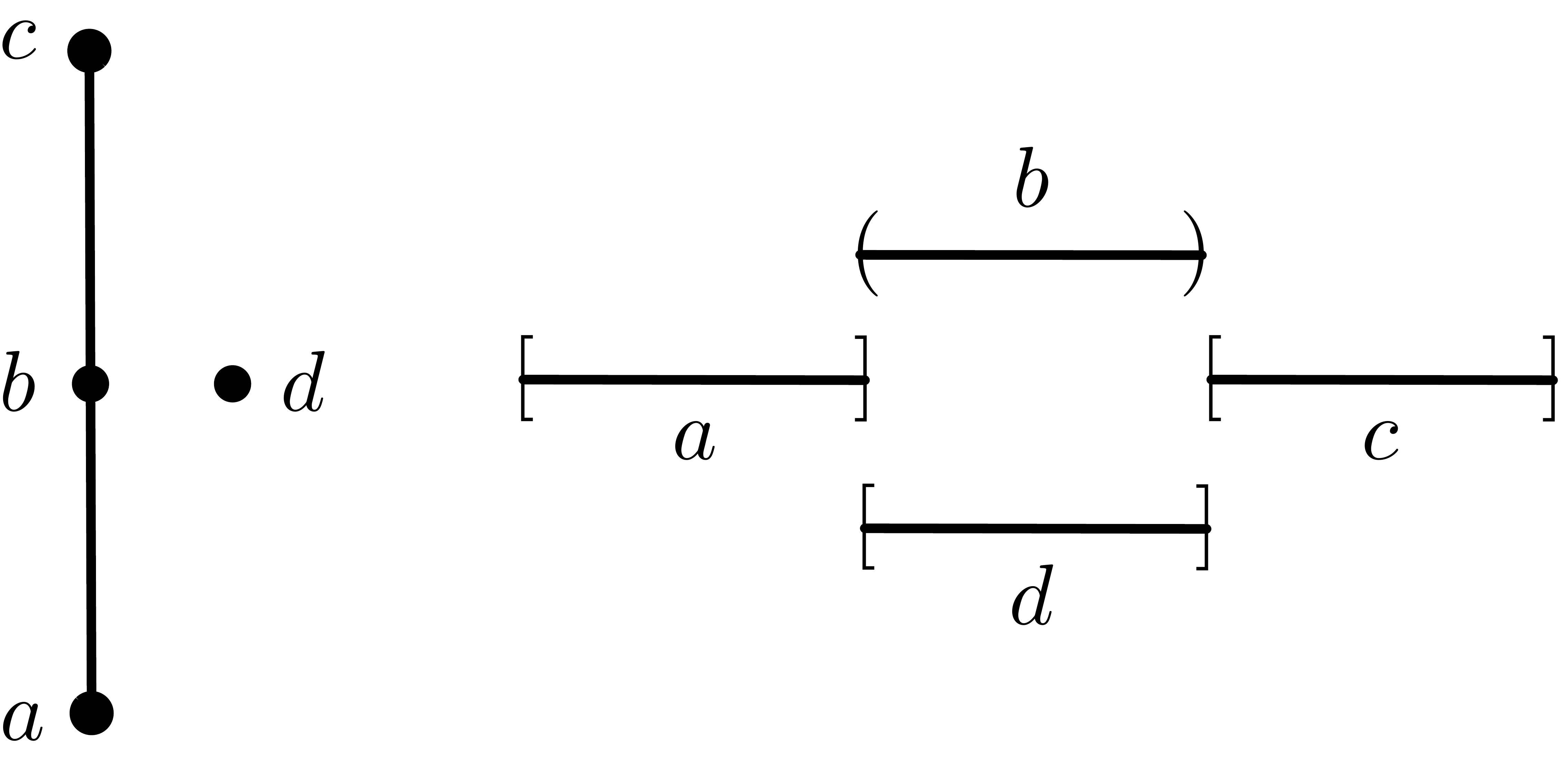}}
  \caption{The order {\threeone}  and its unique unit OC representation.}
  \label{fig-threeone}
\end{figure}

 Two points in a poset are said to be \emph{twins} if they have precisely the same comparabilities.  In an interval representation of a poset, if two points are assigned the same interval then they are twins.  A poset is \emph{twin-free} if no two points are twins.
 
 In this section we describe the set ${\cal F}$ of posets that are forbidden for  twin-free unit OC interval orders.   We begin with the poset {\threeone} with ground set $a,b,c,d$ and comparabilities $a \prec b \prec c$   shown in Figure~\ref{fig-threeone}.  The unique unit OC interval representation of {\threeone} is shown to its right.  In this  and other figures, we shorten the notation by labeling an interval as $x$ instead of $I(x)$.

\begin{figure}
 \begin{picture}(300,100)(0,0)
\thicklines

\put(20,20){\circle*{5}}
\put(20,40){\circle*{5}}
\put(20,60){\circle*{5}}
\put(20,80){\circle*{5}}
\put(40,50){\circle*{5}}
\put(20,20){\line(0,1){60}}
 \put(5,0){\fourone}

\put(80,30){\circle*{5}}
\put(80,50){\circle*{5}}
\put(80,70){\circle*{5}}
\put(95,50){\circle*{5}}
\put(110,50){\circle*{5}}
\put(80,30){\line(0,1){40}}
 \put(60,0){\threeoneone}
 
 \put(150,20){\circle*{5}}
\put(150,40){\circle*{5}}
\put(150,60){\circle*{5}}
\put(150,80){\circle*{5}}
\put(135,65){\circle*{5}}
\put(165,35){\circle*{5}}

\put(150,20){\line(0,1){60}}
\put(135,65){\line(1,1){15}}
\put(150,20){\line(1,1){15}}
 \put(145,0){$Z$}

  \put(220,20){\circle*{5}}
\put(220,80){\circle*{5}}
\put(205,50){\circle*{5}}
\put(235,50){\circle*{5}}
\put(250,50){\circle*{5}}

\put(220,20){\line(1,2){15}}
\put(220,20){\line(-1,2){15}}
\put(205,50){\line(1,2){15}}
\put(235,50){\line(-1,2){15}}
 \put(220,0){$D$}

  \put(300,20){\circle*{5}}
\put(300,50){\circle*{5}}
\put(285,80){\circle*{5}}
\put(315,80){\circle*{5}}
\put(315,50){\circle*{5}}

\put(300,20){\line(0,1){30}}
\put(300,50){\line(1,2){15}}
\put(300,50){\line(-1,2){15}}
 
 \put(295,0){$Y$}

  \end{picture}

\caption{Forbidden posets which, with the dual of $Y$, comprise the set ${\cal F}$.}

\label{fig-forbid}
 \end{figure}
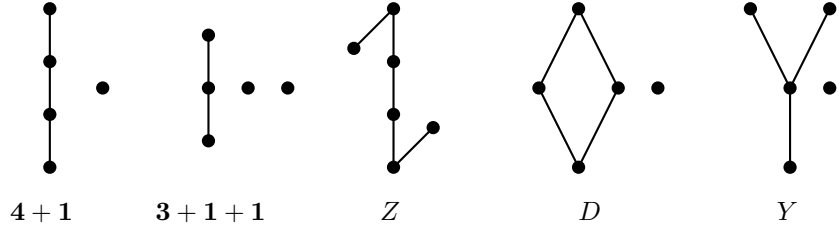
 
In the next proposition, we show that the posets in Figure~\ref{fig-forbid} are forbidden in  twin-free unit OC interval orders.  Each of these posets is self dual, except for $Y$.  Throughout the remainder of this paper we denote the set of posets given in Figure~\ref{fig-forbid} together with the dual of $Y$ by ${\cal F}$.
 
 \begin{prop}
    \label{prop-forbid}
  If $P$ is a twin-free unit OC  interval order then it does not contain any of the five posets in Figure~\ref{fig-forbid}  or the dual of $Y$.
  \end{prop}
  \begin{proof}  Suppose the contrary. We consider  each of the posets of $\cal F$ in turn.
 \begin{enumerate}
  \item Consider the poset {\fourone} with ground set $\{a, b, c, d, x\}$ and comparabilities $a \prec b \prec c \prec d$.  The unique representation of the {\threeone} induced by $\{a, b, c, x\}$ forces $I(b)$ to be an open interval, while the unique representation of the {\threeone} induced by $\{b, c,d,x\}$ forces $I(b)$ to be a closed interval, a contradiction.
  \item Consider the poset {\threeoneone} with ground set $\{a,b,c,x,y\}$ and comparabilities $a \prec b \prec c$.  The unique representations of the two {\threeone}'s induced by $\{a,b,c,x\}$ and $\{a,b,c,y\}$ force $I(x) = I(y)$.  Thus, $x, y$ would be twins in any poset containing a {\threeoneone}.
  \item Consider the poset $Z$ with ground set $\{a,b,c,d,x,y\}$ and comparabilities $a \prec b \prec c \prec d, x \prec d, a \prec y.$  The unique representation of the {\threeone} induced by $\{a,b,c,x\}$ forces $I(b)$ to be an open interval, while the unique representation of the {\threeone} induced by $\{b,c,d,y\}$ forces $I(b)$ to be a closed interval, a contradiction.
  \item Consider the poset $D$ with ground set $\{a,b,c,d,x\}$ and comparabilities $a \prec b \prec d, a \prec c \prec d$.  The unique representations of the {\threeone}'s induced by $\{a,b,d,x\}$ and $\{a,c,d,x\}$ force $I(b)$ and $I(c)$ to be identical intervals.  Thus $b$ and $c$ would be twins in any poset containing $D$.
\item Finally consider the poset $Y$ with ground set $\{a,b,c,d,x\}$ and comparabilities $a \prec d \prec b, a \prec d \prec c$.  The two {\threeone}'s induced by $\{a,d,b,x\}$ and $\{a,d,c,x\}$ force $I(b)$ and $I(c)$ to be identical intervals.  Thus, $b$ and $c$ would be twins in any poset containing $Y$.  The argument for the dual is similar.
   \qed
  \end{enumerate}  

  \end{proof}

Observe that in  the proof of Proposition~\ref{prop-forbid},   the \emph{twin-free} hypothesis is used only in the arguments for the posets {\threeoneone}, $D$, and $Y$.  Indeed, these three posets are unit OC interval orders if twins are allowed.  The arguments for the posets {\fourone} and $Z$ did not use the twin-free hypothesis, and indeed these are not unit OC interval orders even when twins are allowed.
 
 \section{Representing interval orders}
 In this section we describe how to obtain an initial closed interval representation for a given interval order and describe several important properties of this representation.
 
 \begin{defn}
 \label{def-downup}
 {\rm Let $P = (X, \prec)$ be a poset.  For any $x \in X$, the \emph{down set of} $x$, denoted by $D(x)$, is the set $\{y \in X : y \prec x\}$.  Similarly, the \emph{up set of} $x$, denoted by $U(x)$ is the set  $\{y \in X : x \prec y\}$.  We let ${\cal D}$ be the set of all down sets of $P$ and ${\cal U}$ be the set of all up sets of $P$. }
 \end{defn}
 
 The following elementary result gives an alternate characterization of interval orders and appears in  \cite{Tr92}.  We omit the proof.
 
 \begin{prop}
 \label{prop-trotter}
 The following are equivalent for a poset $P = (X, \prec)$.
 \begin{enumerate}
 \item $P$ has no induced {\twotwo}.
 \item The down sets of $P$ are ordered by set inclusion.
 \item The up sets of $P$ are ordered by set inclusion.
 \end{enumerate}
 \end{prop}
 
Theorem~\ref{thm-green} below is based on work by Greenough \cite{Gr76} and shows how to construct a representation of an interval order from its down sets and up sets.  A statement of this result appears in \cite{KeTr10} and an illustrative example is shown in Figure~\ref{fig-interval-rep}.  We include the proof for completeness.
 
 \begin{theorem}
 \label{thm-green}  
Let $P = (X, \prec)$ be an interval order.  Index the down sets and the up sets of $P$ as follows:  
\[\emptyset = D_1 \subset D_2 \subset \ldots \subset D_{|{\cal D}|}  \mbox{\rm{ and }} U_1 \supset U_2 \supset \ldots \supset  U_{|{\cal U}|} = \emptyset. \]

\noindent For each $x \in X$ let $L(x) = i$ where $D(x) = D_i$, and let $R(x) = j$ where $U(x) = U_j$.  Then the following hold:
\begin{enumerate}
\item $|{\cal D}| = |{\cal U}|$.
\item $L(x) \leq R(x)$ for each $x \in X$.
\item The assignment of the interval $I(x) = [L(x), R(x)]$ to each $x$ in $X$ produces an interval representation of $P$.
\end{enumerate}
 \end{theorem}
 \begin{proof}
 Start with a closed interval representation of $P$.  We will modify this representation so that every value that appears as an endpoint of an interval appears as both a left endpoint of an interval and a right endpoint of an (possibly the same) interval.  Note that some intervals will be reduced to single points by this process.  Suppose there is a real number that appears as a left endpoint in the representation but not as a right endpoint.  Let $z$ be the leftmost of these values.  Let $w$ be the smallest number greater than $z$ for which there is an interval whose right endpoint is $w$.  We replace each interval of the form
 $[z,x]$ 
by $[w,x]$ to produce a new set of intervals that provides another closed representation of $P$. The new representation has one less value that appears as a left endpoint but not a right endpoint.

   \begin{figure}

\centering
 {\includegraphics[height=2.75in]{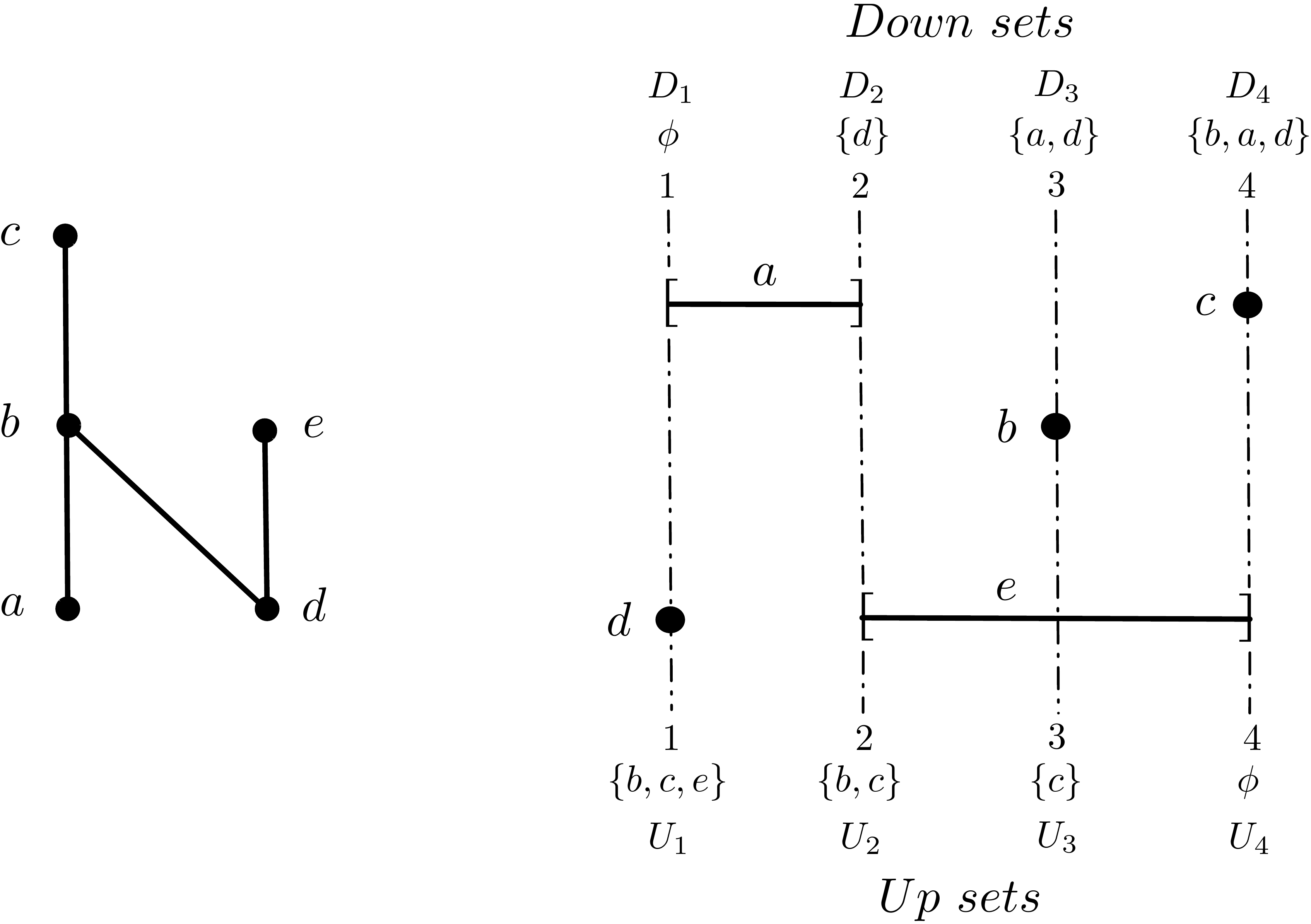}}
  \caption{The poset $N$ and its interval representation produced as in Theorem~\ref{thm-green}.}
   \label{fig-interval-rep}
\end{figure}
 
 Continue sweeping from left to right in this way, modifying the representation until every value that appears as a left endpoint also appears as a right endpoint.  Then sweep from  right to  left and modify the representation further until every value that appears as a right endpoint also appears as a left endpoint.  The result is an interval representation of $P$ in which there are $m$ numbers that appear as endpoints of intervals and each is both a left endpoint of some interval and the right endpoint of some (possibly the same) interval.  Without loss of generality we may assume these $m$ endpoints are labeled $1, 2, \ldots, m$.  Let ${\cal I}$ denote this interval representation.

For each $i$ between 1 and $m$ there exists an $x \in X$ whose left endpoint is $i$ and thus $D(x)$ consists of elements of $X$ (if any) whose intervals lie completely to the left of $i$. Since the endpoints satisfy $1 < 2 < \ldots < m$ and there is a right endpoint at each $i$, these $m$ down sets are distinct and completely ordered by inclusion from smallest to largest.  Thus $D_i = D(x)$ for those $x \in X$ whose left endpoints are labeled $i$.  Furthermore, $|{\cal D}| = m$.  Similarly, $U_i = U(x)$ for those $x \in X$ whose right endpoints are labeled $i$ and $|{\cal U}| = m$.  This proves (1).

Every $x \in X$ whose left endpoint lies on $i$ has $D(x) = D_i$ and thus $L(x) = i$, and every $x \in X$ whose right endpoint lies on $j$ has $U(x) = U_j$ and thus $R(x) = j$.  The intervals in our representation ${\cal I}$ are indeed intervals, thus $L(x) \leq R(x)$ for all $x \in X$, proving (2).

The interval assigned to $x$ in ${\cal I}$ is $[i,j]$ where $D(x) = D_i$ and $U(x) = U_j$, thus the interval is also $[L(x), R(x)]$, proving (3).
 \qed
 \end{proof}
 
 \smallskip
 
Figure~\ref{fig-interval-rep} illustrates applying Theorem~\ref{thm-green}   to  the poset $N$.
 The following remark is a consequence of the proof of Theorem~\ref{thm-green}.
 
 \begin{remark}
 \label{remark-green}
The representation given in Theorem~\ref{thm-green} has the property that  every value that appears as an endpoint of an interval appears as both a left endpoint and a right endpoint. 
 \end{remark}

In  Theorem~\ref{thm-peekers} we will modify the representation given in Theorem~\ref{thm-green} so that each proper inclusion has the useful property that we define in Definition~\ref{def-peekers}.  

\begin{defn}
\label{def-peekers} {\rm Let $P = (X, \prec)$ be an OC interval order and fix an interval representation of it.
For $x,u,v \in X$ with $I(u) \subset I(v)$,
 we say $I(x)$ (or $x$) \emph{peeks into $vu$} if $I(x)$ intersects $I(v)$ but not $I(u)$.  Furthermore, it \emph{peeks into $vu$ from the left} if in addition  $ R(x) \le L(u)$ and \emph{peeks into $vu$ from the right} 
 if $R(u) \le L(x).$
 }
\end{defn}

Figure~\ref{fig-peek}(a) illustrates this definition, where $x$ peeks into $ab$   from the left and $y$ peeks into $ab$ from the right.
 In the following theorem we modify the  representation given in Theorem~\ref{thm-green}  to have distinct endpoints  so that there can be left and right peekers into each proper inclusion.

\begin{theorem}
\label{thm-peekers}
Every twin-free interval order $P = (X, \prec)$ has a closed interval representation satisfying the following {\bf peeking property}:  For each proper inclusion $I(u) \subset I(v)$ there exist $x,y \in X$ so that $x$ peeks into $vu$ from the left and $y$ peeks into $vu$ from the right.
\end{theorem}

\begin{proof}
Given a twin-free interval order $P$, use the assignment of intervals in Theorem~\ref{thm-green} to produce an interval representation of $P$.  Observe that each endpoint is an integer.  Let $I(v) = [j, j+k]$ be the interval assigned to $v$ where $j,k \in {\integers}$ and $k \ge 0$.  We modify this representation to produce one with distinct endpoints. 

Let $\hat{I}(v) = \left[ j - \frac{1}{k+3}, \  j + k + \frac{1}{k+3} \right]$ for each $v \in X$.  Since $P$ is twin-free the endpoints in $\{ \hat{I}(v) : v \in X \}$ are distinct by construction.  One can verify that $\hat{I}(v)$ properly contains $\hat{I}(u)$ if and only if  
$I(v), I(u)$ have distinct endpoints and $I(v)$ properly contains $I(u)$.  

Suppose $\hat{I}(v)$ properly contains $\hat{I}(u)$.  Then $I(v)$ properly contains $I(u)$ and they have distinct endpoints.  By Remark~\ref{remark-green} there exist $x,y \in X$ where the right endpoint of $I(x)$ equals the left endpoint of $I(v)$ and the left endpoint of $I(y)$ equals the right endpoint of $I(v)$.  It follows that, in both the original and the modified representations of $P,$ $x$ peeks into $vu$ from left and $y$ peeks into $vu$ from the the right.
\qed
\end{proof}

\section{The Main Theorem}

We are now ready to state and prove our main theorem.

\begin{theorem}
 Let $P = (X, \prec)$ be a twin-free interval order.  The following are equivalent:

\begin{enumerate}[(1)]  
\item $P$ is a unit OC interval order.
  
\item $P$ is a strict OC interval order.
  
\item $P$ has no induced poset from the forbidden set $\cal F$ consisting of the five   orders in Figure~\ref{fig-forbid} and the dual of $Y$.

\end{enumerate}
  
  \label{big-thm}
  \end{theorem}

  \begin{proof} The proof that (1) $\Rightarrow (3)$ follows from Proposition~\ref{prop-forbid}. 
We next show that (2) $\Rightarrow (1) $.  

Let $P = (X, \prec)$ be a strict OC interval order and fix a strict OC interval representation ${\cal I}$ of $P$.  Take the closure $\overline{I(v)}= [L(v),R(v)]$ of each interval in this representation and remove duplicates, i.e., say two elements are equivalent if their intervals have the same closure and take one representative from each equivalence class.  
Let $X^\prime \subset X$ be the resulting set of elements.

The intervals $\overline{I(v)}$ for $ v \in X^\prime$ determine a proper representation, $\overline{\cal I}$,  of an interval order $P^\prime$.
Apply the Bogart-West procedure in \cite{ BoWe99} to this proper representation to obtain a unit representation ${\cal I'}$ of $P^\prime$ in which element $v \in X'$
 is assigned interval  $I'(v) = [L'(v),R'(v)]$.  As observed  in \cite{RaSz13}, this construction  satisfies

\begin{equation}
\tag{$\ast$}
R(u) = L(v)  \hbox{ if and only if } R'(u) = L'(v), \hbox{for all } u, v\in X'.
\end{equation}

Now  extend  ${\cal I'}$ to $X$ as follows. For each $v\in X$ whose representative in $P'$ is $w$,   let $I''(v)$ have the same endpoints as $I'(w)$ and let $I''(v)$ be closed if and only if $I(v)$ is closed.   Each interval  in $\{I''(v):v \in X\}$ has unit length, and using $(\ast)$ it is straightforward to show that this set of intervals gives a representation of $P$.  \qed 

\end{proof}

In the remaining subsections, we   prove (3) $\Rightarrow (2)  $  of Theorem~\ref{big-thm}.  We begin by describing certain properties that our initial closed representation will satisfy.  

\subsection{Properties of the initial representation}

 \begin{prop}  
\label{claims}  
Let $P = (X, \prec)$ be a twin-free,   ${\cal F}$-free interval order.  Then there exists   a closed interval representation ${\cal I} = \{I(x) : x \in X \}$  of $P$ satisfying the   following:
 
 \begin{enumerate}[(1)]

\item No interval   strictly contains two other intervals.

\item No interval  is strictly contained in  two other intervals.

\item If $I(u) \subset I(v)$
then there are unique  $x, y \in X$ so that $x$ peeks into $vu$ from the left and $y$ peeks into $vu$ from the right.
\end{enumerate}

\end{prop}

\begin{proof} Use the construction in the proof of 
  Theorem~\ref{thm-peekers} to form   a closed interval representation ${\cal I} = \{I(x) : x \in X \}$  of $P$ satisfying the peeking property.  We show this representation also satisfies (1), (2) and (3).

{\bf Proof of (1): }  Suppose there exist elements 
  $u\ne v$ of $X$ whose intervals are strictly contained in $I(w)$.  
 We will show that every possible configuration of these intervals leads to a contradiction. 
 
 First consider the case of strictly nested intervals:  $I(u) \subset I(v) \subset I(w)$.  Since ${\cal I}$ satisfies the peeking property, there exist elements  $x,y$ such that $x$ peeks into $vu$ from the left, and $y$ peeks into $vu$ from the right.  In this case,  $x \prec u \prec y$  and $x, u, y, v, w$ induce a {\threeoneone} in $P$, a contradiction since $P$ is ${\cal F}$-free.  
 
 Next suppose that $u$ and $v$ are not nested.  We may assume without loss of generality that $I(u)$ has the leftmost left endpoint and $I(v) $ has the rightmost right endpoint of all the intervals  strictly contained in $I(w)$. 
 
 There exist elements $x,y$ such that $x$ peeks into $wu$ on the left and  $y$  peeks into $wv$ on the right.  If $I(u) \cap I(v) = \emptyset$ then   $x \prec u \prec v \prec y$ and $w,x,u,v,y$ induce a {\fourone} in $P$, a contradiction.  Hence $I(u) $ and $ I(v)$ intersect and the elements $w,x,u,v,y$ induce the forbidden poset $D$ in $P$, also a contradiction. 

\smallskip

{\bf Proof of (2):} Suppose there exist $u,v,w$ so that $I(u)$ is strictly contained in both $I(v)$ and $I(w)$.  Without loss of generality 
 we may assume that $L(w) < L(v) < L(u)$.  By (1), no three intervals can be nested so $R(w) < R(v)$.  Then there exist elements $x,y$ where $x$ peeks into $vu$ from the left and $y$ peeks into $wu$ from the right.  Now the elements $w,v,x,u,y$ induce a {\threeoneone} in $P$,  a contradiction.
 
 \smallskip

{\bf Proof of (3):} 
Since $\cal I$ satisfies the peeking property, there exist elements $x,y$ where $x$
peeks into $vu$ from the left and  $y$ peeks into $vu$ from the right. Suppose there exists a second $x' \in X$ that peeks into $vu$ from the left.  Then $u,v,x,x^\prime,y$  induce  the dual of $Y$  in $P$, a contradiction.     A similar argument yields forbidden configuration $Y$ when two intervals peek into $vu$ from the right.  \qed

\end{proof}

  \subsection{Completing the proof of Theorem~\ref{big-thm}}

In this section we complete the proof of Theorem~\ref{big-thm} by proving $(3) \Rightarrow (2)$.  
Let $P = (X,\prec)$ be an interval order with closed interval representation ${\cal I} = \{I(z) : z \in X\}$.  Suppose $I(u)$ is properly contained in $I(v)$.  We say that $I(u)$ is the \emph{inner} interval of this proper inclusion and $I(v)$ is the \emph{outer interval.}  If $x$ peeks into $vu$ from the left and we redefine $I(x)$ to be $\hat I(x) = [L(x), L(v)]$, we say that $I(x)$ (or $x$) is \emph{retracted to the left}. Similarly,  if $y$ peeks into $vu$ from the right and we redefine $I(y)$ to be $\hat I(y) = [R(v), R(y)]$, we say that $y$ is \emph{retracted to the right}.  Finally, we say that the inner interval $I(u)$ is \emph{expanded (to meet $v$)} if it is redefined to be $\hat I(u) = (L(v), R(v))$.  Figure~\ref{fig-peek}
illustrates these concepts.  
We use these retractions and expansions to convert $\cal I$ into a representation with no strict inclusions.  
In this section, we denote the left and right endpoints of interval $\hat I(v)$ respectively by $\hat L(v)$ and $\hat R(v)$.

\begin{figure}
  \begin{center}
  {\includegraphics[height=1in]{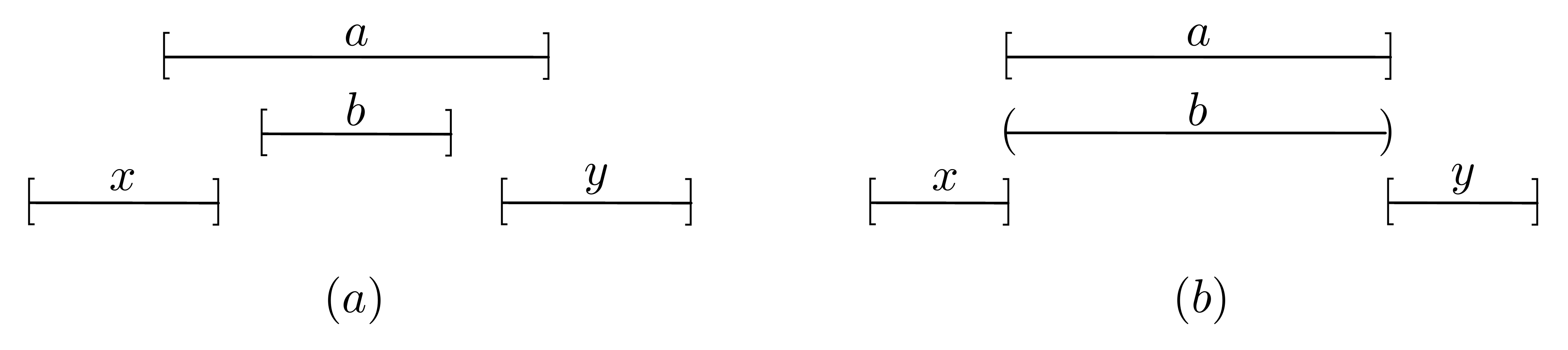}}
  \end{center}
   \caption{(a) $x$ peeks into $ab$   from the left and $y$ peeks into $ab$ from the right. (b)  $x$ is retracted to the left, $y$ is retracted to the right and   $b$ is expanded to meet $a$. }
   \label{fig-peek}

\end{figure}
  
  \begin{prop} Let $P =  (X, \prec)$ be a twin-free, ${\cal F}$-free interval order with closed interval representation  ${\cal I} = \{I(z) : z \in X\}$ satisfying the three conclusions of Proposition~\ref{claims}.   For any proper inclusion,   if the left peeker is retracted  to the left and the right peeker is retracted to the right, then the  resulting set of closed intervals is  also an interval representation of $P$ and  the new representation also satisfies the three conclusions of Proposition~\ref{claims} and has the same proper inclusions as $\cal I$.

  \label{prop-retract}
  \end{prop}
  \begin{proof}    
  Let  $I(u) \subset I(v)$ be a proper inclusion in $\cal I$.  By our hypothesis, there exists a unique $x \in X$ for which $I(x)$
    peeks into $vu$ from the left, and thus  $L(v) \le R(x) < L(u)$.   We retract $x$ to the left and define a new set of intervals ${\hat{\cal I}}$ as follows:  $\hat I(x) = [L(x), L(v)]$ and $\hat I(w) = I(w)$  for $w \neq x$.  Let 
    $\hat P = (X, \hat \prec)$ be the interval order represented by ${\hat{\cal I}}$. 
   We will prove that the orders $P$ and $\hat P$ are equal by showing that any two elements $a,b \in X$ have the same relation in   $\hat P$ as in $P$.
   
   Since retracting $I(x)$ cannot create a new incomparability, we need only consider the case in which $a \parallel b$ in $P$ but $a \  \hat \prec \ b$ in $\hat P$.  Since $a \parallel b$ in $P$, we know $L(b) \le R(a)$, and since the only change made in defining $\hat{\cal I}$ is to $R(x)$, we must have $a = x$. Then 
$$L(v) = \hat R(x) = \hat R(a) < \hat L(b) = L(b)  \leq R(a)  < L(u) < R(v).$$
   
If $R(b) \leq R(v)$ then $I(v)$ properly contains both $I(u)$ and $I(b)$, contradicting  conclusion (1) of Proposition~\ref{claims}.  If $R(b) > R(v)$ then $I(u)$ is properly contained in both $I(v)$ and $I(b)$, contradicting conclusion (2) of  Proposition~\ref{claims}.  We conclude that $a\ \hat \parallel \ b$ and thus that $P = \hat P$.

Next we show that no   proper inclusions were created or lost  in retracting $x$ to the left.   If a proper inclusion were created or lost, there would exist $a \in X$ so that $I(x) \not\subset I(a)$ and $\hat I(x)  \subset 
\hat I(a)$ (or $I(a)  \subset I(x)$ and $\hat I(a)  \not\subset 
\hat I(x)$).  In either case, both $a$ and $x$ peek into $vu$ from the left in $P$, contradicting conclusion (3) of Proposition~\ref{claims}.

After retraction, $x$ continues to peek into $vu$ from the left and so $\hat P = P$. Thus the new representation $\hat {\cal I}$ continues to satisfy the three conclusions of Proposition 13.
Finally, note that the argument is similar if $y$ peeks into $vu$ from the right. \qed
  \end{proof}

 \begin{prop} Let $P =  (X, \prec)$ be a twin-free  interval order  with no induced $Z$ and let ${\cal I} = \{I(z) : z \in X\}$ be a closed interval representation of it  satisfying the three conclusions of Proposition~\ref{claims}, where all peekers  have been retracted in repeated applications of Proposition~\ref{prop-retract}.   If for each proper inclusion, the inner interval is expanded to meet the outer interval, then the resulting set of intervals $\hat {\cal I}$ is a strict OC interval representation of $P$.
 \label{prop-expand}
 \end{prop} 
  
 \begin{proof}
 Let $\hat P = (X, \hat \prec)$ be the OC interval order represented by ${\hat{\cal I}}$.    We will prove that the orders $P$ and $\hat P$ are equal by showing that any two elements $a,b \in X$ have the same relation in $\hat P$ as in $P$.  For a contradiction, suppose there exist $a,b \in X$ with one relation in $P$ and a different relation in $\hat P$.  Since expanding intervals cannot  cause a new comparability, it must be the case that $a$ and $b$ are comparable in $P$ but incomparable in $\hat P$.    Without loss of generality, we may assume $a \prec b$ in $P$.    There are two cases to consider depending on whether one or both  of $I(a)$, $I(b)$ are inner intervals.
 
 \smallskip
\emph{Case 1}. Only one  of $I(a)$, $I(b)$ is an inner interval.  By symmetry, we may assume it is  $I(b)$. Then $I(b) \subset I(w)$ for some $w \in X$ and $b$ is expanded so that  $\hat{\cal I}(b) = (L(w), R(w))$.

 Then $L(w) = \hat L(b) \le \hat R(a) = R(a)$.
 Thus in $\cal I$, $a$ is the unique element of $X$ that peeks into $wb$ from the left and so, by Proposition~\ref{prop-retract},  $I(a)$ was retracted to the left and $R(a) = L(w)  $.
 But since $\hat I(a) = I(a)$ is closed and $\hat I(b)$ is open, this implies that $a\  \hat \prec \  b$, a contradiction. 
  
 
 \smallskip
 \emph{Case 2}. Both $I(a), I(b)$ are inner intervals. Thus there are $v,w \in X$ such that $I(a) \subset I(v)$ and $I(b) \subset I(w)$, and both $a,b$ are expanded  to open intervals with the same endpoints as $I(v), I(w)$ respectively.   Since  $a\  \hat \parallel \ b$ we have $L(w) = \hat L(b) < \hat R(a) = R(v)$.  If $L(w) \le L(v)$ or $R(v) \ge R(w)$, the intervals for $a,v,w$ or $b,w,v$ would be nested in $\hat {\cal I}$, contradicting Proposition~\ref{claims}.  Thus 
   $L(v) < L(w) < R(v) < R(w)$.  
    
 If $R(v) < L(b)$ then $v$ peeks into $wb$ from the left and $v$ would have been retracted in Proposition~\ref{prop-retract}, making $R(v) = L(w)$, a contradiction. So $R(v) \ge L(b)$, and similarly $L(w) \le R(a)$.  By Proposition~\ref{claims}, there are unique elements $x,y \in X$, where $x$ peeks into $va$ from the left and   $y$ peeks into $wb$ on the right. Then the elements $x,a,b,y,v,w$ 
 induce the forbidden graph $Z$ of Figure~\ref{fig-forbid} in $P$, contrary to our hypothesis.
 $\qed$

 \end{proof}
 
 \bigskip
 
The proof that  (3) implies (2)   in Theorem~\ref{big-thm} now follows from   Propositions~\ref{claims}, \ref{prop-retract}, and~\ref{prop-expand}.

\section{Recognition and realization algorithm}
In this section we present an efficient algorithm for recognizing the class of strict OC interval orders.  Given a twin-free  poset $P=(X, \prec)$ as input, the algorithm returns a strict OC  interval representation in the case that $P$ belongs to the class.  Otherwise, it returns a poset from the set ${\cal F}$ induced in $P$ or indicates that the poset is not an interval order.  

The algorithm proceeds in three stages and,  if the algorithm has not terminated, the output of one stage is the input to the next.  In Stage 1, either a $\twotwo$ is discovered in $P$ and the algorithm terminates, or a closed interval representation $\cal I$  of $P$ is constructed that has distinct endpoints and satisfies the peeking property of Theorem~\ref{thm-peekers}.   In addition, a storage matrix $A$ is created and an inclusion matrix $B$ is initialized.  

In Stage 2, either a forbidden poset  in the set \{$\fourone$, \  $\threeoneone$,   $Y$,   dual of $Y$,   $D$\}
is discovered in $P$ and the algorithm terminates, or otherwise we can conclude that   $\cal I$ also satisfies the three conclusions of Proposition~\ref{claims}.  Matrix $B$ is updated so that it records all strict inclusions of intervals in $\cal I$.  
In Stage 3, either a forbidden poset $Z$ is discovered in $P$ and the algorithm terminates, or otherwise the representation $\cal I$ is modified to be a unit OC interval representation of $P$.

\medskip

\noindent
{\bf Algorithm:  OC Interval Order:}

\smallskip
\noindent
{\bf Input:}  A twin-free poset   $P=(X, \prec)$. 

\smallskip
\noindent
{\bf Output:}   Either returns a $\twotwo$, or a poset from the set $\cal F$ induced in $P$, or  constructs  a strict OC interval representation of $P$.

We begin with  Stage 1,   which     constructs  a closed interval representation when given an interval order and initializes the necessary data structures.

\smallskip

\noindent {\bf Stage 1: Find Initial Representation.}

\smallskip


Calculate the down sets of $P$ and order them $D_1, D_2, D_3, \ldots, D_k$ so that $|D_1| \le |D_2| \le |D_3| \le \cdots \le  |D_k|$.  Check whether $D_i$ is strictly contained in $D_{i+1}$, for $1 \le i \le k-1$.    If there is an $i$ for which $D_i \not\subset D_{i+1}$ then $P$ contains an induced $\twotwo$ which can be found as follows.  Find $x,y $ with $x \in D_i \backslash D_{i+1}$ and $y \in D_{i+1} \backslash D_{i}$, and $a,b$ for which $D_i = D(a)$ and $D_{i+1} = D(b)$.  Then the elements $x,y,a,b$ induce a $\twotwo$ in $P$, and thus $P$ is not an OC interval order by Proposition~\ref{prop-OC}.  \emph{Return these four elements and terminate.  }

Otherwise, by Propositions~\ref{prop-OC} and \ref{prop-trotter}, $P$ is an interval order,  there are  exactly $k$ up sets, and they can be ordered by set inclusion.  Calculate the up sets of $P$ and order them $U_1, U_2, U_3, \ldots , U_k$  by size so that $U_1 \supset U_2 \supset U_3 \supset \cdots \supset U_k$.
Use the intervals given in Theorem~\ref{thm-green}(3) to get an initial interval representation of $P$.  Modify these intervals as in the proof of Theorem~\ref{thm-peekers} to obtain a closed interval representation ${\cal I} = \{I(x) : x \in X\}$ of $P$ that has distinct endpoints and satisfies the peeking  property of Theorem~\ref{thm-peekers}. 

 Next, create the storage matrix $A$.  Sort the $2|X|$ interval endpoints of the representation $\cal I$ from smallest to largest.   The \emph{index} of an endpoint is its position on this list, thus each element $x \in X$ is assigned two indices, those of its left and right endpoints.  
 Record information about this representation in a matrix $A$ that has three rows and $2|X|$ columns, where information about the endpoint $e_j$  with index $j$ is stored in column $j$:
  
  \smallskip
 
\noindent
 \indent $A_{1j} = x $, where $x$ is the   element  of $  X$  that is assigned  index $j$; \\
 \indent    $A_{2j} = L $  if $e_j$ is a left endpoint of $I(x)$ and  $A_{2j} = R $ otherwise; \\
   \indent   $A_{3j}$ is the index of the other endpoint of $I(x)$.  

\noindent
\begin{table}
\begin{center}
\begin{tabular}{|c||c|c|c|c|c|c|c|c|c|c|}  \hline
Index & 1 & 2 & 3 & 4 & 5 & 6 & 7 & 8 & 9 & 10 \\ \hline\hline
 Element &d &a &d &e &a &b &b &c &e &c  \\ \hline
 L/R &L &L &R &L &R &L &R &L &R &R  \\ \hline
 Other end & 3& 5& 1& 9& 2& 7& 6& 10& 4& 8 \\ \hline
\end{tabular}
\end{center}
\caption{The storage matrix $A$ for the poset $N$ in Figure~\ref{fig-interval-rep}.}
\label{matrix-A-table}
\end{table}

Table~\ref{matrix-A-table} shows the matrix $A$ for the poset $N$ of Figure~\ref{fig-interval-rep}.  The entries in  $A$ will not change as the algorithm proceeds.
In addition,   create a matrix $B$ with three rows and $|X|$ columns, one column for each $x \in X$.  The first row of   $B$  provides the  indices corresponding to each $x \in X$.  Fill in these  
    entries using information in array $A$ as follows:

\smallskip    
 \indent  $B_{1x} = \  <i,j>$,  where $i$ is the index of $L(x)$ and $j$ is the index of $R(x)$. 
 \smallskip
      
    The  entries  in the remaining two rows will later record the sets of elements whose intervals contain $I(x)$ and are contained in $I(x)$.  These entries are initialized to be empty.  This ends Stage 1.  Return the interval representation $\cal I$ and the matrices $A$ and $B$.

\smallskip
    \noindent
{\bf Output of Stage 1:}   A $\twotwo$ induced in $P$ or a closed interval representation $\cal I$ of $P$ that has distinct endpoints and satisfies the peeking  property of Theorem~\ref{thm-peekers}, 
 and matrices $A$ and $B$.

    \smallskip

During Stage 2, we record proper inclusions of intervals in $\cal I$  in matrix $B$ so that  by the end of Stage 2, the entries satisfy the following.

\smallskip
$B_{2x} = \{y \in X: I(x) \subseteq I(y\}$ (i.e.,   elements whose intervals contain $I(x)$) \\
\indent $B_{3x} = \{z \in X: I(z) \subseteq I(x)\}$ (i.e., elements whose intervals  $I(x)$ contains) \\ 

 If our representation fails to satisfy one of the three conclusions of Proposition~\ref{claims},
this will be discovered in Stage 2  and a forbidden graph from $\cal F$ produced.  We may need to find up to two left peekers into a proper inclusion and   this is achieved by the   subroutine  \emph{Locate Peekers},  which appears just after  Stage 2.  Finding right peekers is analogous.

\medskip

 \noindent {\bf Stage 2:  Identify inclusions and peekers.}


   Initialize an empty queue $Q$.   For $ j = 1$ to $2|X|$ perform the following operations: 

\medskip
 
 If $A_{2j} = L$, add index $j$ onto the back of $Q$. Increment $j$.

\medskip

  Otherwise,  index $j$ represents the right endpoint of an interval $I(x)$  and the index $i$ of $L(x)$ is currently on $Q$.  Using matrices $A$ and $B$, scan from the front of $Q$ to find $i$. One of three things can occur.

\begin{enumerate}
\item  Index  $i$ is at the front of the 	queue.  

In this case,   remove $i$ from the queue and increment $j$.  No updates are made to matrix $B$.

\item  At least two indices are on the queue in front of $i$.    In this case,     find the forbidden poset $\threeoneone$ in $P$   as follows and terminate.

Let $k_1,k_2 \in Q$ with $k_1 < k_2 < i$ and let $v_1,v_2$ be the elements of $X$ whose left endpoints have   indices $k_1,k_2$ respectively.  Let $k_3$ be the index of the right endpoint of $I(v_1)$ and $k_4$ be the index of the right endpoint  of $I(v_2)$.   
 Then $I(x)$ is contained in both $I(v_1) $ and $I(v_2)$ and $L(v_1) < L(v_2)$.  Locate a left peeker $u_1$ into $v_2x$.  Locate a right peeker $u_2$ into    $v_1x$ (if $k_3 < k_4$) or into $v_2x$ (if $k_4 < k_3$).  The elements $u_1,x,u_2,v_1,v_2$ induce a  $\threeoneone$ in $P$.  \emph{Return these elements and terminate.}  If  for all $j$, the algorithm does not terminate in this case, then the representation satisfies conclusion (2) of Proposition~\ref{claims}.

\item  Exactly one index is on the queue in front of $i$.  

Let $k$ be the index in front of $i$ on the queue and let $w \in X$ have index $k$.  Update $B_{2x}$ to include $w$ and update 
$B_{3w}$ to include $x$.  

If $|B_{3w}| \ge 2$, we find  a forbidden poset  in $P$ as follows and terminate.  Let $y \in B_{3w}$ with $y \neq x$.  Since the algorithm scans endpoints from left to right, we know $L(y) < L(x)$ and since we are in case 3, we know $I(x) \not\subseteq I(y)$.  Locate a left peeker $u_1$ into $wy$ and a right peeker $u_2$ into $wx$.   The elements $u_1,y,x,u_2, w$ induce a forbidden poset in $P$:  either a $\fourone$ 
if $I(x), I(y)$ do not intersect
or the poset $D$  
if they do.
\emph{Return these   elements and terminate.}  If  for all $j$, the algorithm does not terminate in this case, then the representation satisfies conclusion (1) of Proposition~\ref{claims}.

Otherwise, $|B_{3w}| = 1$.  If there are two right peekers into $wx$ (or two left peekers),     find a forbidden poset in $P$ as follows and terminate.  Locate two right peekers $u_1,u_2$ and one left peeker $u_3$ into $wx$.  The elements $u_1,u_2, x, u_3, w$ induce the  poset $Y$ in $P$.  Similarly, if there are two left peekers, we discover the dual of $Y$.    In either of these cases, \emph{return these elements and terminate.}  Otherwise, remove index $i$ from the queue (even though it is not at the front) and increment $j$.
  If for all $j$,  the algorithm does not terminate in this case, then the representation satisfies conclusion (3) of Proposition~\ref{claims}.   

\end{enumerate}

\medskip
\noindent
{\bf Output of Stage 2:}  Either a forbidden poset from the set \{$\fourone$, $\threeoneone$, $Y$, dual of $Y$, $D$\} or else  matrix $A$, the updated matrix $B$, and  a closed interval representation $\cal I$ of $P$ that satisfies the three conclusions of Proposition~\ref{claims}.
 \medskip

\medskip
\noindent
   {\bf Subroutine Locate Peekers:}  

\noindent
{\bf Input: }
 $x,y \in X$ with $I(x) \subset I(y)$.

\smallskip

\noindent
{\bf Output:}  a set of at most two left peekers into $yx$.

\smallskip

Use matrix $B$ to find left indices $i$  of $x$ and $k$ of $y$.  By the construction of $Q$ we know $k < i$.    Any index between $k$ and $i$ representing a right endpoint corresponds to a left peeker into $yx$.    Start at $i$ and scan back towards $k$ to locate up to two such indices.   Use matrix $A$ to locate the element(s) that are assigned to these indices,  \emph{return these element(s) and terminate the subroutine.}

\medskip

At the end of Stage 2, the algorithm has either identified and returned a forbidden configuration and terminated, or it has  identified and recorded in $B$ all proper inclusions.  Furthermore, the algorithm has verified that the closed interval representation $\cal I$ satisfies the conclusions of Proposition~\ref{claims}.    Thus each entry in the bottom two rows of $B$ is either empty or contains exactly  one element.  Furthermore, there is exactly one left peeker and one right peeker into each inclusion.  

 In Stage 3,  make two passes through the indices, retracting   peekers in the first pass and expanding each inner interval in the second.   The left indices of outer intervals are placed on a queue $\hat{Q}$, allowing us to keep track of outer intervals that intersect.   The algorithm either verifies that the result is a strict OC representation of $P$ or returns the forbidden poset $Z$   induced in $P$.  Finding a forbidden  $Z$ is achieved by the subroutine \emph{Locate Z,} which appears just after    Stage 3.

  \medskip

\noindent  {\bf Stage 3:     Retracting Peekers and Expanding Inner Intervals.}

\smallskip


\noindent
{\bf 3a. Retract Peekers:}  
Initialize an empty queue $\hat{Q}$.  For $ j = 1$ to $2|X|$ perform the following operations: 

Let $v = A_{1j}$.    If $B_{3v} = \emptyset$ (i.e., $I(v)$ does not contain another interval) then increment $j$.  
Otherwise, $I(v)$ contains exactly one other interval $I(u)$, where $u \in B_{3v}$.  

If $A_{2j} = L$, the next step either finds  a forbidden graph $Z$ induced in $P$, or else places $j$ on the back of $\hat{Q}$.  
Each index $i$ currently on $\hat{Q}$ is the left endpoint of an outer interval $I(v')$ with $L(v') < L(v)$.   For each such $i$,  use matrices $A$  and $B$ to locate $v'$ and $u'$ where $I(u') \subset I(v')$.    If for some $i$,  $I(v') \cap I(u) \neq \emptyset$ but $I(u) \cap I(u') = \emptyset$,  use the subroutine \emph{Locate $Z$} to find a forbidden poset $Z$ in $P$ and terminate.    Otherwise, place $j$ on the back of $\hat{Q}$.  
 If for all $j$, the algorithm does not terminate in this case, then we show $P$ does not contain an induced $Z$ as follows.   Suppose poset $Z$ with comparabilities $a \prec b \prec c \prec d$, $a \prec y$, $x \prec d$ is induced in $P$.  Then for  any closed interval representation of $P$,  when $j$ is the index of $L(y)$, the index $i$  of $L(x)$ is still on $\hat{Q}$ and the subroutine \emph{Locate $Z$}  is called with the elements $u=c$, $v = y$, $u'=b$ and $v' = x$.  
 
If $A_{2j} = R$, locate the index of the left endpoint of $v$ and remove it from $\hat{Q}$.  Locate the left peeker $x$ and the right peeker $y$ into $vu$.  Retract these peekers by redefining $R(x) := L(v)$ and $L(y) := R(v)$.  Increment $j$.    If the algorithm has not terminated, then the revised set of intervals provides a representation of $P$ that satisfies the hypotheses of Proposition~\ref{prop-expand}.

\medskip
\noindent
{\bf 3b. Expand Inner Intervals:}  \ For each $u \in X$, if $B_{2u} \neq \emptyset$ then $B_{2u} = \{v\}$ for some $v \in X$.     
    Expand $u$ to meet $v$ by redefining $L(u) := L(v)$ and $R(u) := R(v)$ and making the interval for $u$ open.    By Proposition~\ref{prop-expand}, the resulting set of intervals is a strict OC interval representation of $P$.
    
\medskip
\noindent
{\bf Output of Stage 3:}  Either a forbidden graph $Z$ induced in $P$ or otherwise a strict  OC interval representation of $P$.
\medskip

 \noindent
 {\bf Subroutine:  Locate $Z$}
 
 \noindent
 {\bf Input:}  
  Elements $u,v,u',v' \in X$ with $I(u) \subset I(v)$, $I(u') \subset I(v')$,    $L(v') < L(v)$,  $I(v') \cap I(u) \neq \emptyset$  and    $I(u') \cap I(u) = \emptyset$.
  
  \noindent
  {\bf Output: }  Elements that induce the forbidden poset $Z$ in $P$.
   
  Observe that $I(v) \cap I(u') \neq \emptyset$, for otherwise there would be  two right peekers into $v'u'$, namely $v$ and $u$, a contradiction.   Locate   the left peeker $w_1$ into  $v'u'$ and    the right peeker $w_2$ into $vu$.  Note that $I(w_1) \cap I(v) = \emptyset$ for otherwise $I(v)$ would contain both $I(u)$ and $I(u')$, a contradiction.  Similarly, $I(w_2) \cap I(v') = \emptyset$.   The elements $w_1,u',u,w_2,v',v$ induce the forbidden poset $Z$ in $P$.  \emph{Return these six elements and terminate the subroutine.}
  
  \medskip
  It is easy to verify the following run time for our algorithm.
\begin{remark}{\rm  Algorithm OC Interval Order runs in time that is quadratic in $|X|$.}

\end{remark}

\section{Forbidden posets characterization when twins are allowed}

Let $\cal T$ be the set of posets $\fourone, D^*, Z, Y^*, Y^{**}$ shown in Figure~\ref{fig-more-forbid}, together with their duals.  In this section, we prove the following theorem that characterizes the class of unit OC interval orders.

\begin{figure}
 \begin{picture}(300,100)(0,0)
\thicklines

\put(20,20){\circle*{5}}
\put(20,40){\circle*{5}}
\put(20,60){\circle*{5}}
\put(20,80){\circle*{5}}
\put(35,50){\circle*{5}}
\put(20,20){\line(0,1){60}}
 \put(8,0){\fourone}


 \put(160,20){\circle*{5}}
\put(160,40){\circle*{5}}
\put(160,60){\circle*{5}}
\put(160,80){\circle*{5}}
\put(145,65){\circle*{5}}
\put(175,35){\circle*{5}}

\put(160,20){\line(0,1){60}}
\put(145,65){\line(1,1){15}}
\put(160,20){\line(1,1){15}}
 \put(155,0){$Z$}

  \put(80,20){\circle*{5}}
\put(80,80){\circle*{5}}
\put(65,50){\circle*{5}}
\put(95,50){\circle*{5}}
\put(110,50){\circle*{5}}
  \put(110,20){\circle*{5}}

\put(80,20){\line(1,2){15}}
\put(80,20){\line(-1,2){15}}
\put(65,50){\line(1,2){15}}
\put(95,50){\line(-1,2){15}}
\put(95,50){\line(1,-2){15}}
\put(75,0){$D^*$}

  \put(220,20){\circle*{5}}
\put(220,50){\circle*{5}}
\put(205,80){\circle*{5}}
\put(235,80){\circle*{5}}
\put(235,50){\circle*{5}}
\put(228,66){\circle*{5}}

\put(220,20){\line(0,1){30}}
\put(220,50){\line(1,2){15}}
\put(220,50){\line(-1,2){15}}
\put(235,80){\line(0,-1){30}}
\put(217,0){$Y^*$}

  \put(290,20){\circle*{5}}
\put(290,50){\circle*{5}}
\put(275,80){\circle*{5}}
\put(305,80){\circle*{5}}
\put(305,50){\circle*{5}}
\put(320,50){\circle*{5}}

\put(290,20){\line(0,1){30}}
\put(290,50){\line(1,2){15}}
\put(290,50){\line(-1,2){15}}
\put(305,80){\line(0,-1){30}}
\put(287,0){$Y^{**}$}


  \end{picture}

\caption{Forbidden posets which, with their duals,  comprise the set $\cal T$.}

\label{fig-more-forbid}
 \end{figure}
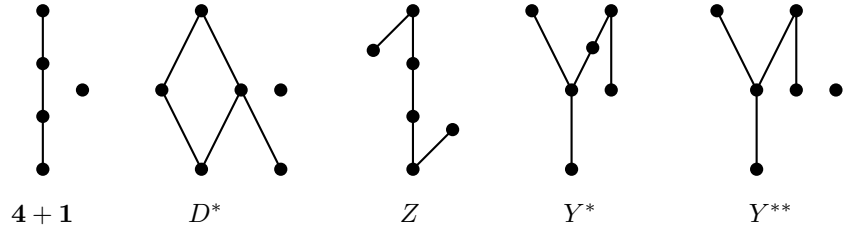

%

\begin{theorem}
An interval order is a unit OC  interval order if and only it has no induced poset from the forbidden set ${\cal T}$.

\label{thm-forbid}
\end{theorem}

\begin{proof}  In this proof we use the notation  and build on the arguments introduced in the proof of Proposition~\ref{prop-forbid}.  First we show that the posets in $\cal T$ are not unit OC interval orders.  In the proof of Proposition~\ref{prop-forbid}, we showed that the posets $\fourone$ and $Z$ are not unit OC interval orders even if twins are allowed.  Any unit OC interval
  representation of poset  $D, Y$,  or the dual of $Y$ requires the intervals for  $b$ and $c$ to be identical.  The posets $D^*, Y^*, Y^{**}$ and their duals  each contain an additional element $y$  comparable to one of $b, c$ but not the other.  Thus in any     unit OC interval representation of  $D^*, Y^*,Y^{**}$ or their duals,  the intervals for $b$ and $c$ would not be identical.  If the interval for $y$ were removed, the result would be a unit OC interval representation of $D, Y$, or the dual of $Y$ in which the intervals for $b$ and $c$ are not identical, a contradiction.

  For the converse, let $P = (X,\prec)$, and suppose $P$ is an interval order which is not a unit OC interval order.  We will show it contains a poset from $\cal T$ as an induced subposet.
  Let $P'$ be the poset obtained from $P$ by removing duplicate points, that is, from each set of points with the same comparabilities,   take one representative.  Poset $P'$ is induced in $P$ but is twin-free.  If $P'$ had a unit OC interval representation, we could give identical intervals to duplicate points and form a unit OC interval representation of $P$, a contradiction.  Thus $P'$ is a twin-free interval order that is not a unit OC interval order.  By Proposition~\ref{prop-forbid}, we know $P'$ contains an induced poset from $\cal F$.  If $P'$ contains $\fourone$ or $Z$ then $P'$ (and hence $P$) contains a poset from $\cal T$ as desired.  We give the argument for $P'$ containing the poset $D$ below.  The arguments for $P'$ containing $\threeoneone$ or $Y$ are similar.

    Suppose $P'$ contains $D$ with ground set $\{a,b,c,d,x\}$ and comparabilities $a \prec b \prec d$, $a \prec c \prec d$.  Since $b$ and $c$ are not twins in $P$, there must exist an element $y$ comparable to one, say $c$, but not the other, $b$.  First suppose $y \prec c$.   If $y \prec x$,   poset $P$ contains an induced $\twotwo$,  and if $x \prec y$ then transitivity implies $x \prec c$, each giving a contradiction.      Thus $x \parallel y$.  If $a \prec y$,   poset  $P$ contains an induced $\fourone$, which is an element of $\cal T$.  The remaining possibility is  $a \parallel y$ and this results in the poset $D^*$ induced in $P$.  If instead  $c \prec y$,  we get similar contradictions or the dual of $D^\ast$.   \qed
\end{proof}
\smallskip
 
\noindent \emph{Acknowledgements:}  The authors are grateful to Tom Trotter, Mitch Keller, and David Howard for helpful discussions of the proof of Theorem~\ref{thm-green}.

\end{document}